\documentclass[10pt,dvipsnames]{amsart}
\usepackage{hyperref}
\usepackage{amssymb, amscd, amsmath, amsfonts, amsthm}
\usepackage{stmaryrd}
\usepackage{verbatim}
\def\unprotectedboldentry#1{\textcolor{Red}{\textbf{#1}}}
\def\boldentry{\protect\unprotectedboldentry}
\usepackage{tikz}
\usetikzlibrary{calc}
\usepackage{ifthen}
\newcommand{\tikztableau}[2][scale=0.6,every node/.style={font=\small}]{
    \def\newtableau{#2}
    \begin{array}{c}
    \begin{tikzpicture}[#1]
    \coordinate (x) at (-0.5,0.5);
    \coordinate (y) at (-0.5,0.5);
    \foreach \row in \newtableau {
        \coordinate (x) at ($(x)-(0,1)$);
        \coordinate (y) at (x);
        \foreach \entry in \row {
            \ifthenelse{\equal{\entry}{X}}
               {
                \node (y) at ($(y) + (1,0)$) {};
                \fill[color=gray!10] ($(y)-(0.5,0.5)$) rectangle +(1,1);
                \draw[color=gray] ($(y)-(0.5,0.5)$) rectangle +(1,1);
               }
               {
                \ifthenelse{\equal{\entry}{\boldentry X}}
                   {
                    \node (y) at ($(y) + (1,0)$) {};
                    \fill[color=gray] ($(y)-(0.5,0.5)$) rectangle +(1,1);
                    \draw ($(y)-(0.5,0.5)$) rectangle +(1,1);
                   }
                   {
                    \node (y) at ($(y) + (1,0)$) {\entry};
                    \draw ($(y)-(0.5,0.5)$) rectangle +(1,1);
                   }
               }
            }
        }
    \end{tikzpicture}
    \end{array}}

\def\uu{\mathbf u}
\def\hh{\mathbf h}
\def\bb{\mathbf b}
\def\AA{\mathbb{A}}
\def\BB{\mathbb{B}}

\def\upgraph{\mathcal G^{\uparrow}}
\def\downgraph{\mathcal G^{\downarrow}}
\def\ascomp{\operatorname{ascomp}}

\def\Strong{\operatorname{Strong}}
\newcommand\kschur[1][k]{s^{(#1)}}
\newcommand\dualkschur{\widetilde F}
\newcommand\nckschur[1][k]{\mathfrak{s}^{(#1)}}
\newcommand\id{\textrm{\rm id}}

\def\weakstrip{\leadsto}
\def\strongstrip{\rightarrowtriangle}
\def\size{\operatorname{size}}

\newcommand{\ribbonfcn}[1]{\overline{s_{#1}}}
\def\kBounded{\mathcal B^{(k)}}
\def\Langle{\left\langle}
\def\Rangle{\right\rangle}

\newtheorem{Theorem}{Theorem}[section]
\newtheorem{Proposition}[Theorem]{Proposition}
\newtheorem{Corollary}[Theorem]{Corollary}
\newtheorem{Lemma}[Theorem]{Lemma}
\newtheorem{Example}[Theorem]{Example}
\newtheorem{Remark}[Theorem]{Remark}

\allowdisplaybreaks
\begin{document}

\title{Pieri operators on the affine nilCoxeter algebra}
\author{Chris Berg \and Franco Saliola \and Luis Serrano}
\date{\today}

\begin{abstract}
We study a family of operators on the affine nilCoxeter algebra. We use these
operators to prove conjectures of Lam, Lapointe, Morse, and Shimozono regarding
strong Schur functions. 
\end{abstract}

\maketitle
%\tableofcontents

\section{Introduction}

The $k$-Schur functions of Lapointe, Lascoux and Morse \cite{LLM03} first arose
in the study of Macdonald polynomials. Since then, their study has flourished;
see for instance \cite{LM03, LM05, LM07, LS07, LLMS10, Lam10} and the
references therein. This is due, in part, to an important geometric
interpretation of the Hopf algebra $\Lambda_{(k)}$ of $k$-Schur functions and
its dual Hopf algebra $\Lambda^{(k)}$: these algebras are isomorphic to the
homology and cohomology of the affine Grassmannian in type~A \cite{Lam08}.
Under this isomorphism, the $k$-Schur functions map to the Schubert basis of
the homology and the dual $k$-Schur functions (also called the affine Schur
functions) map to the Schubert basis of the cohomology.

An important problem in the theory of $k$-Schur functions is to find a
$k$-Littlewood--Richardson rule, namely, a combinatorial interpretation for the
(nonnegative) coefficients in the expansion
\begin{equation}\label{eq:klr}
\kschur_{\mu} \kschur_{\nu} = \sum_{\lambda} c^{\lambda, (k)}_{\mu, \nu} \kschur_{\lambda}.
\end{equation}
The $c^{\lambda, (k)}_{\mu, \nu}$ are called the
$k$-Littlewood--Richardson-coefficients, and are of high relevance in
combinatorics and geometry. It was proved by Lapointe and Morse \cite{LM08}
that special cases of these coefficients yield the 3-point Gromov--Witten invariants. The 3-point Gromov--Witten invariants are
the structure constants of the quantum cohomology of the Grassmanian; they
count the number of rational curves of a fixed degree in the Grassmannian.

As an approach to finding the $k$-Littlewood--Richardson coefficients, Lam
\cite{Lam06} identified $\Lambda_{(k)}$ with the affine Fomin--Stanley subalgebra
$\BB$ of the affine nilCoxeter algebra $\AA$ of the affine symmetric group $W$.
Specifically, he constructed a family of elements $\nckschur_{\lambda}\in\BB$
that map under this isomorphism to the $k$-Schur functions
$\kschur_{\lambda}$. Furthermore, he proved \cite[Proposition 42]{Lam06} that
finding the $k$-Littlewood--Richardson rule is equivalent to finding the
expansion of the $\nckschur_{\lambda}$ in the ``standard basis'' $\uu_w$ of
$\AA$. Explicitly, he proved that the coefficients in \eqref{eq:klr} appear as
coefficients in the expansion
\begin{equation}\label{eq:expansion}
\nckschur_\lambda = \sum_{w\in W} d_\lambda^w \uu_w.
\end{equation}

In this article, we develop a family of operators on $\AA$, which will facilitate the study of the $\nckschur_\lambda$, and we prove certain conjectures regarding a family of functions that generalize the $k$-Schur functions $\kschur_\lambda$. Each of these is described in more detail below.

\subsection{The Pieri operators}
Lam, Lapointe, Morse, and Shimozono \cite{LLMS10} constructed a labelled
directed graph $\downgraph$ on the elements of $W$, which encompasses the
strong order in $W$. In this article we study the operators on $\AA$ induced by
the Pieri operators of $\downgraph$ in the spirit of \cite{BMSvW00}. In Section \ref{sec:operators}, we develop
the main properties of these operators. More specifically, in Theorem
\ref{restriction} we prove that these operators are determined by their
restriction to $\BB$, in Theorem \ref{PieriRulePerp} we determine this
restriction, and in Theorem \ref{cor:commute} we prove that the operators
commute pairwise.

\subsection{Properties of strong Schur functions}
Lam, Lapointe, Morse, and Shimozono \cite{LLMS10} generalized the
$\kschur_\lambda$ to a larger set of functions called the strong Schur
functions $\Strong_{u/v}$, where $u$ and $v$ is any pair of elements in $W$. In
Section \ref{sec:strong}, we use the Pieri operators to prove a series of
conjectures of Lam, Lapointe, Morse, and Shimozono
\cite[Conjecture~4.18]{LLMS10} regarding these functions. Specifically,
\begin{enumerate}
\item[(a)] in Theorem \ref{thm:LLMS1} we prove that the $\Strong_{u/v}$ are
    symmetric functions;
\item[(b)] in Theorem \ref{thm:LLMS2} we prove that they belong to the algebra
    $\Lambda_{(k)}$; and
\item[(c)] in Theorem \ref{thm:LLMS3} we describe the coefficient of
    $\kschur_\lambda$ in $\Strong_{u/v}$, when $u$ and $v$ are $0$-Grassmannian
    elements, in terms of the structure constants of the cohomology ring of the
    affine flag variety.
\end{enumerate}
Note that (c) provides a combinatorial description of the \emph{skew} $k$-Schur
functions.

\subsection{Acknowledgements}
We would like to thank Nantel Bergeron, Sergey Fomin, Thomas Lam, Jennifer Morse, Anne Schilling, and Mike
Zabrocki for helpful discussions.

This research was facilitated by computer exploration using the open-source
mathematical software system \texttt{Sage}~\cite{sage} and its algebraic
combinatorics features developed by the \texttt{Sage-Combinat} community
\cite{sage-combinat}.

\section{Background and Notation}

\subsection{Affine symmetric group}
Fix a positive integer $k$. Let $W$ denote the affine symmetric group with
simple generators $s_0, s_1, \ldots, s_k$. There is an
interpretation of $W$ as the group of permutations $w: \mathbb Z \to \mathbb Z$
satisfying $w(i+k+1) = w(i) + k+1$ for all $i\in\mathbb Z$ and $\sum_{i=1}^{k+1}
w(i)=\sum_{i=1}^{k+1} i$. Let $t_{i,j}$ be the element of $W$ that interchanges the
integers $i$ and $j$ and fixes all integers not congruent to $i$ or $j$ modulo
$k+1$.

Let $W_0$ denote the subgroup of $W$ generated by $s_1, \dots, s_k$ and let
$W^0$ denote the set of minimal length coset representatives of $W/W_0$.
Elements of $W^0$ are called \emph{affine Grassmannian elements} or
\emph{$0$-Grassmannian elements}. There are bijections between $0$-Grassmannian
elements, $k$-bounded partitions, and $(k+1)$-cores. We will not review these
here, but refer the reader to \cite{LM05}. For a $k$-bounded partition $\lambda$,
we let $w_\lambda$ denote the corresponding element of $W^0$. Let $\kBounded$
denote the set of $k$-bounded partitions.

\subsection{Affine nilCoxeter algebra}
Let $\AA$ denote the \emph{affine nilCoxeter algebra} of $W$: this is the
algebra generated by $\uu_0, \uu_1, \dots, \uu_k$ with relations:
\begin{gather*}
    \uu_i^2 = 0 \text{ for all $i$}; \\
    \uu_i \uu_{i+1} \uu_i = \uu_{i+1} \uu_i \uu_{i+1}
        \text{ with $i+1$ taken modulo $k+1$}; \\
    \uu_i \uu_j = \uu_j \uu_i
        \text{ if $i - j \neq \pm1$ modulo $k+1$.}
\end{gather*}
It follows that a basis of $\AA$ is given by the elements
$\uu_w = \uu_{s_{i_1}}\uu_{s_{i_2}}\cdots\uu_{s_{i_l}}$,
where $w = s_{i_1}s_{i_2}\cdots s_{i_l}$ is a reduced word for $w \in W$.
We define an inner product on $\AA$ by
$\langle \uu_v, \uu_w \rangle_\AA = \delta_{u,v}$.

\subsection{Affine Fomin--Stanley subalgebra}
An element $w \in W$ is said to be \emph{cyclically decreasing} if there exists
a reduced factorization $s_{i_1}\cdots s_{i_j}$ of $w$ satisfying: each letter
occurs at most once; and, for all $m$, if $s_m$ and $s_{m+1}$ both appear in
the reduced factorization, then $s_{m+1}$ precedes $s_m$. If $D \subsetneq \{0,
1, \dots, k\}$, then there is a unique cyclically decreasing element $w_D$ with
letters $\{s_d : d \in D\}$. Let $\uu_D = \uu_{w_D}$ denote the corresponding
basis element of $\AA$. For $i \in \{0, 1, \dots, k\}$, let
\begin{align*}
\hh_i = \sum_{\substack{D \subset I \\ |D| = i}} \uu_D \in \AA.
\end{align*}
By a result of Thomas Lam \cite{Lam06}, the elements $\{\hh_i\}_{i \leq k}$ commute
and freely generate a subalgebra $\BB$ of $\AA$ called the \emph{affine
Fomin--Stanley subalgebra}. The elements $\hh_\lambda = \hh_{\lambda_1} \dots
\hh_{\lambda_t}$, for all $k$-bounded partitions $\lambda = (\lambda_1, \dots,
\lambda_t)$, form a basis of $\BB$.

\subsection{Symmetric functions}
\label{kspace}
Let $\Lambda$ denote the ring of symmetric functions. For a partition
$\lambda$, we let $m_\lambda$, $h_\lambda$, $e_\lambda$, $p_\lambda$,
$s_\lambda$ denote the monomial, homogeneous, elementary, power sum and Schur
symmetric function, respectively, indexed by $\lambda$. Each of these families
forms a basis of $\Lambda$. We recall the following change of bases formulae:
\begin{align*}
    h_\mu = \sum_{\lambda} K_{\lambda, \mu} s_\lambda
    \qquad \text{and} \qquad
    s_\lambda = \sum_{\mu} K_{\lambda, \mu} m_\mu
\end{align*}
where $K_{\lambda, \mu}$, called the \emph{Kostka number}, is the number
of semistandard tableaux of shape $\lambda$ and content $\mu$ \cite{Sta99}.

Let $\Lambda_{(k)}$ denote the subalgebra of $\Lambda$ generated by $h_0$,
$h_1$, $\dots$, $h_k$. The elements $h_\lambda$ with $\lambda_1 \leq k$ form a
basis of $\Lambda_{(k)}$. Let $\Lambda^{(k)} = \Lambda/I_k$ denote the quotient
of $\Lambda$ by the ideal $I_k$ generated by $m_\lambda$ with $\lambda_1 > k$.
The equivalence classes in $\Lambda^{(k)}$ of the elements $m_\lambda$ with
$\lambda_1 \leq k$ form a basis of $\Lambda^{(k)}$. 

The \emph{Hall inner product} of symmetric functions is defined by
$$
 \langle h_\lambda, m_\mu \rangle_\Lambda = 
 \langle s_\lambda, s_\mu \rangle_\Lambda = 
 \delta_{\lambda,\mu}.
$$
Observe that every element of the ideal $I_k$ is orthogonal to every element of
$\Lambda_{(k)}$ with respect to this inner product. Hence, it induces a
pairing $\langle \cdot, \cdot \rangle$ between $\Lambda_{(k)}$ and
$\Lambda^{(k)}$. In particular, $\langle f, g \rangle = \langle f, \widetilde g
\rangle$ for $f \in \Lambda_{(k)}$, $g \in \Lambda^{(k)}$ and any preimage
$\widetilde g$ of $g$ under the quotient map $\Lambda \to \Lambda^{(k)}$.
For an element $f$ in $\Lambda^{(k)}$, write $f^\perp : \Lambda_{(k)} \to
\Lambda_{(k)}$ for the linear operator that is adjoint to multiplication by $f$
with respect to $\Langle \cdot, \cdot \Rangle$.

\subsection{Affine Schur functions}
The affine Schur functions form a distinguished basis of $\Lambda^{(k)}$.
For $w \in W$, the \emph{affine Stanley symmetric function} is defined as
\begin{gather}\label{eq:affstanley}
    \dualkschur_w = \sum_{\lambda \in \kBounded}
    \big\langle \hh_\lambda, \uu_w \big\rangle_\AA \, m_\lambda,
\end{gather}
where $m_\lambda$ is the monomial symmetric function indexed by $\lambda$.
These functions are elements of $\Lambda^{(k)}$, but they are not linearly
independent. For a $k$-bounded partition $\lambda$, let $\dualkschur_\lambda =
\dualkschur_{w_\lambda}$, where $w_\lambda$ denotes the $0$-Grassmannian
element corresponding to $\lambda$. The functions $\dualkschur_\lambda$ are
called \emph{affine Schur functions} (or \emph{dual $k$-Schur functions}) and
they form a basis of $\Lambda^{(k)}$. See for instance \cite{Lam06, LM08}.

\subsection{$k$-Schur functions}
\label{ss:kschurs}
The $k$-Schur functions are a distinguished basis of $\Lambda_{(k)}$.
They are defined as the duals of the affine Schur functions with respect
to the inner product $\langle \cdot, \cdot \rangle$ on $\Lambda_{(k)} \times
\Lambda^{(k)}$. That is, they satisfy
$\langle \kschur_\lambda, \dualkschur_\mu \rangle = \delta_{\lambda,\mu}$
for all $k$-bounded partitions $\lambda$ and $\mu$.
Equivalently, they are uniquely defined by the \emph{$k$-Pieri rule}:
\begin{align*}
    h_i \kschur_\lambda = \sum \kschur_\nu
\end{align*}
where the sum ranges over all $k$-bounded partitions $\nu$ such that
$w_\nu w_\lambda^{-1}$ is cyclically decreasing of length $i$.
It follows from duality and \eqref{eq:affstanley} that
\begin{align*}
    h_\mu = \sum_{\lambda\in\kBounded}
    \big\langle \hh_\lambda, \uu_{w_\mu} \big\rangle_\AA
    \, \kschur_\lambda.
\end{align*}

\subsection{Noncommutative $k$-Schur functions}
\label{ss:noncommkschurs}
The algebras $\Lambda_{(k)}$ and $\BB$ are isomorphic with
isomorphism given by $h_\lambda \mapsto \hh_\lambda$.
We denote by $\nckschur_\lambda$ the image of the
$k$-Schur function $\kschur_\lambda$ under this isomorphism.
In the literature, $\nckschur_\lambda$ is called
a \emph{noncommutative $k$-Schur function}.
They have the following expansion \cite[Proposition~42]{Lam06}:
\begin{gather}
    \label{eq:kschurexpansioninA}
    \nckschur_\lambda = \sum_{w \in W}
    \Langle \kschur_\lambda, \dualkschur_w \Rangle
    \, \uu_w.
\end{gather}
That is, the coefficient of $\uu_w$ in $\nckschur_\lambda$ is equal
to the coefficient of $\dualkschur_\lambda$ in $\dualkschur_w$:
\begin{gather}
    \label{eq:relationbetweenpairings}
    \Langle \kschur_\lambda, \dualkschur_w \Rangle
    =
    \Langle \nckschur_\lambda, \uu_w \Rangle_\AA
\end{gather}
and so
\begin{gather*}
    \dualkschur_w = \sum_{\lambda\in\kBounded}
    \Langle \nckschur_\lambda, \uu_w \Rangle_\AA \dualkschur_\lambda.
\end{gather*}
Consequently, $\nckschur_\lambda$ contains exactly one term $\uu_w$ with $w \in
W^0$ and its coefficient is $1$. Furthermore, if $\sum_{w} c_w \uu_w$ is known
to lie in $\BB$, then
$\sum_{w} c_w \uu_w = \sum_{\lambda} c_{w_\lambda} \nckschur_{\lambda}$.

\section{Definition of the operators}

In this section, we define operators on the affine nilCoxeter algebra $\AA$. The definitions are dependent upon the combinatorics introduced by Lam, Lapointe, Morse, and Shimozono in \cite{LLMS10}.

\subsection{Up operators}

Define an edge-labelled oriented graph $\upgraph$ with vertex set $W$:
there is an edge from $v$ to $w$ labelled by $i$ whenever
$\ell(w) = \ell(v) + 1$
and
$s_iv = w$.
(See Figure \ref{fig:upgraph}.)
So, $\upgraph$ is the graph for the \emph{(left) weak order} on $W$.

\begin{figure}[htb]
\begin{center}
\begin{tikzpicture}[scale=0.60,>=latex,line join=bevel,]
  \node (u0*u1*u2) at (74.359bp,212bp) [draw,draw=none] {$s_{0}s_{1}s_{2}$};
  \node (u2*u1*u0) at (368.36bp,212bp) [draw,draw=none] {$s_{2}s_{1}s_{0}$};
  \node (u2*u1) at (418.36bp,144bp) [draw,draw=none] {$s_{2}s_{1}$};
  \node (u1*u2*u1*u0) at (238.36bp,280bp) [draw,draw=none] {$s_{1}s_{2}s_{1}s_{0}$};
  \node (u1*u2) at (104.36bp,144bp) [draw,draw=none] {$s_{1}s_{2}$};
  \node (u1*u0) at (327.36bp,144bp) [draw,draw=none] {$s_{1}s_{0}$};
  \node (u2) at (118.36bp,76bp) [draw,draw=none] {$s_{2}$};
  \node (u1) at (386.36bp,76bp) [draw,draw=none] {$s_{1}$};
  \node (u0*u1*u0) at (484.36bp,212bp) [draw,draw=none] {$s_{0}s_{1}s_{0}$};
  \node (u1*u2*u0) at (156.36bp,212bp) [draw,draw=none] {$s_{1}s_{2}s_{0}$};
  \node (u0) at (272.36bp,76bp) [draw,draw=none] {$s_{0}$};
  \node (1) at (272.36bp,7bp) [draw,draw=none] {$1$};
  \node (u0*u1*u2*u0) at (87.359bp,280bp) [draw,draw=none] {$s_{0}s_{1}s_{2}s_{0}$};
  \node (u2*u0) at (197.36bp,144bp) [draw,draw=none] {$s_{2}s_{0}$};
  \node (u0*u1) at (464.36bp,144bp) [draw,draw=none] {$s_{0}s_{1}$};
  \node (u0*u2) at (38.359bp,144bp) [draw,draw=none] {$s_{0}s_{2}$};
  \node (u0*u2*u1*u0) at (470.36bp,280bp) [draw,draw=none] {$s_{0}s_{2}s_{1}s_{0}$};
  \node (u1*u2*u1) at (238.36bp,212bp) [draw,draw=none] {$s_{1}s_{2}s_{1}$};
  \node (u0*u2*u1) at (571.36bp,212bp) [draw,draw=none] {$s_{0}s_{2}s_{1}$};
  \node (u0*u2*u0) at (19.359bp,212bp) [draw,draw=none] {$s_{0}s_{2}s_{0}$};
  \definecolor{strokecol}{rgb}{0.0,0.0,0.0};
  \pgfsetstrokecolor{strokecol}
  \draw [->] (u2*u1*u0) -- (u0*u2*u1*u0);
  \draw [->] (u1*u2*u0) -- (u0*u1*u2*u0);
  \draw [->] (u1*u2*u0) -- (u1*u2*u1*u0);
  \draw [->] (u2*u1*u0) -- (u1*u2*u1*u0);
  \draw [->] (u2*u0) -- (u1*u2*u0);
  \draw [->] (u1*u0) -- (u2*u1*u0);
  \draw [->] (u0) -- (u1*u0);
  \draw [->] (u0) -- (u2*u0);
  \draw [->] (1) -- (u0);
  \draw [->] (u1*u0) -- (u0*u1*u0);
  \draw [->] (u0*u1) -- (u0*u1*u0);
  \draw [->] (u1) -- (u0*u1);
  \draw [->] (u2*u0) -- (u0*u2*u0);
  \draw [->] (u0*u2) -- (u0*u2*u0);
  \draw [->] (u2) -- (u0*u2);
  \draw [->] (1) -- (u2);
  \draw [->] (1) -- (u1);
  \draw [->] (u1) -- (u2*u1);
  \draw [->] (u2) -- (u1*u2);
  \draw [->] (u2*u1) -- (u1*u2*u1);
  \draw [->] (u1*u2) -- (u1*u2*u1);
  \draw [->] (u2*u1) -- (u0*u2*u1);
  \draw [->] (u1*u2) -- (u0*u1*u2);
  %\tikzstyle{every node}=[font=\tiny,fill=white!10]
  %\draw [->] (u2*u1*u0) to node[] {$0$} (u0*u2*u1*u0);
  %\draw [->] (u1*u2*u0) to node[] {$0$} (u0*u1*u2*u0);
  %\draw [->] (u1*u2*u0) to node[] {$2$} (u1*u2*u1*u0);
  %\draw [->] (u2*u1*u0) to node[] {$1$} (u1*u2*u1*u0);
  %\draw [->] (u2*u0) to node[] {$1$} (u1*u2*u0);
  %\draw [->] (u1*u0) to node[] {$2$} (u2*u1*u0);
  %\draw [->] (u0) to node[] {$1$} (u1*u0);
  %\draw [->] (u0) to node[] {$2$} (u2*u0);
  %\draw [->] (1) to node[] {$0$} (u0);
  %\draw [->] (u1*u0) to node[] {$0$} (u0*u1*u0);
  %\draw [->] (u0*u1) to node[] {$1$} (u0*u1*u0);
  %\draw [->] (u1) to node[] {$0$} (u0*u1);
  %\draw [->] (u2*u0) to node[] {$0$} (u0*u2*u0);
  %\draw [->] (u0*u2) to node[] {$2$} (u0*u2*u0);
  %\draw [->] (u2) to node[] {$0$} (u0*u2);
  %\draw [->] (1) to node[] {$2$} (u2);
  %\draw [->] (1) to node[] {$1$} (u1);
  %\draw [->] (u1) to node[] {$2$} (u2*u1);
  %\draw [->] (u2) to node[] {$1$} (u1*u2);
  %\draw [->] (u2*u1) to node[] {$1$} (u1*u2*u1);
  %\draw [->] (u1*u2) to node[] {$2$} (u1*u2*u1);
  %\draw [->] (u2*u1) to node[] {$0$} (u0*u2*u1);
  %\draw [->] (u1*u2) to node[] {$0$} (u0*u1*u2);
\end{tikzpicture}
\end{center}
\caption{A subgraph of $\upgraph$ for $k=2$; \textit{cf.} Figure \ref{fig:downgraph}.}
\label{fig:upgraph}
\end{figure}

A \emph{weak strip} of length $j$ from $w$ to $v$, denoted by $w
\weakstrip v$, is a pair of elements $w, v \in W$ such that $w$ precedes
$v$ in weak order and $vw^{-1}$ is a cyclically decreasing word of
length $j$.

For any non-negative integer $j$, define a linear operator $U_j : \AA
\to \AA$ by
\begin{gather*}
 U_j(\uu_w)
 = \sum_{ w \weakstrip v \atop \size(w \weakstrip v) = j } \uu_{v}
 = \hh_j \uu_w
\end{gather*}
where the sum ranges over all weak strips of length $j$ that begin at
$w$. Equivalently, $U_j$ is multiplication on the left by $\hh_j$.
\begin{Example} 
    With $k=2$:
    $U_{1}(\uu_0) = \uu_2\uu_0 + \uu_1\uu_0$
    and
    $U_{2}(\uu_0) = \uu_0\uu_2\uu_0 + \uu_2\uu_1\uu_0.$
\end{Example}

\subsection{Down operators}

% SAGE CODE TO CONSTRUCT THE GRAPH:
% for y in x.bruhat_lower_covers():
%     (i,j) = to_transposition(y.inverse()*x)
%     S = straddling_transpositions((i,j), n)
%     for (a,b) in S:
%         yb = apply_affine_permutation(y,b)
%         G.add_edge(x,y,yb)
Define a second edge-labelled oriented graph $\downgraph$, the
\emph{marked strong order} graph, with vertex set $W$: there is an edge
from $x$ to $y$ labelled by $y(j)=x(i)$ whenever $\ell(x) = \ell(y)+1$ and
there exists $i \leq 0 < j$ such that $y \, t_{i,j} = x$.

\begin{Example} $(k=2)$
There are two edges from $x = s_0s_1s_2s_0$ to $y = s_1s_2s_0$ since $y^{-1} x$
can be written as $t_{i,j}$ with $i \leq 0 < j$ in two ways: $y^{-1} x =
t_{-4,1} = t_{-1,4}$. These edges are labelled by $y(1) = -2$ and $y(4) = 1$.
See Figure \ref{fig:downgraph}.
\end{Example}

\begin{Remark}
\cite{LLMS10} defined a similar graph except that they oriented their edges in
the opposite direction and labelled the edges by the pair $(i,j)$:
they write $y \buildrel{i,j}\over{\longrightarrow} x$ whereas we write $x
\buildrel{y(j)}\over{\longrightarrow} y$; and they call our label $y(j)$ the
\emph{marking} of the edge.
\end{Remark}

\begin{figure}[htb]
\begin{center}
\begin{tikzpicture}[scale=0.60,>=latex,line join=bevel,]
  \node (u0*u1*u2) at (74.359bp,212bp) [draw,draw=none] {$s_{0}s_{1}s_{2}$};
  \node (u2*u1*u0) at (368.36bp,212bp) [draw,draw=none] {$s_{2}s_{1}s_{0}$};
  \node (u2*u1) at (418.36bp,144bp) [draw,draw=none] {$s_{2}s_{1}$};
  \node (u1*u2*u1*u0) at (238.36bp,280bp) [draw,draw=none] {$s_{1}s_{2}s_{1}s_{0}$};
  \node (u1*u2) at (104.36bp,144bp) [draw,draw=none] {$s_{1}s_{2}$};
  \node (u1*u0) at (327.36bp,144bp) [draw,draw=none] {$s_{1}s_{0}$};
  \node (u2) at (118.36bp,76bp) [draw,draw=none] {$s_{2}$};
  \node (u1) at (386.36bp,76bp) [draw,draw=none] {$s_{1}$};
  \node (u0*u1*u0) at (484.36bp,212bp) [draw,draw=none] {$s_{0}s_{1}s_{0}$};
  \node (u1*u2*u0) at (156.36bp,212bp) [draw,draw=none] {$s_{1}s_{2}s_{0}$};
  \node (u0) at (272.36bp,76bp) [draw,draw=none] {$s_{0}$};
  \node (1) at (272.36bp,7bp) [draw,draw=none] {$1$};
  \node (u0*u1*u2*u0) at (87.359bp,280bp) [draw,draw=none] {$s_{0}s_{1}s_{2}s_{0}$};
  \node (u2*u0) at (197.36bp,144bp) [draw,draw=none] {$s_{2}s_{0}$};
  \node (u0*u1) at (464.36bp,144bp) [draw,draw=none] {$s_{0}s_{1}$};
  \node (u0*u2) at (38.359bp,144bp) [draw,draw=none] {$s_{0}s_{2}$};
  \node (u0*u2*u1*u0) at (470.36bp,280bp) [draw,draw=none] {$s_{0}s_{2}s_{1}s_{0}$};
  \node (u1*u2*u1) at (238.36bp,212bp) [draw,draw=none] {$s_{1}s_{2}s_{1}$};
  \node (u0*u2*u1) at (571.36bp,212bp) [draw,draw=none] {$s_{0}s_{2}s_{1}$};
  \node (u0*u2*u0) at (19.359bp,212bp) [draw,draw=none] {$s_{0}s_{2}s_{0}$};
  \draw [black,->] (u2*u0) ..controls (177.61bp,127bp) and (150.43bp,103.6bp)  .. (u2);
  \definecolor{strokecol}{rgb}{0.0,0.0,0.0};
  \pgfsetstrokecolor{strokecol}
  \tikzstyle{every node}=[font=\footnotesize]
  \draw (175.36bp,110bp) node {$1$};
  \draw [black,->] (u1*u2*u0) ..controls (186.42bp,201.68bp) and (202.25bp,195.37bp)  .. (215.36bp,188bp) .. controls (228.38bp,180.68bp) and (228.95bp,174.58bp)  .. (242.36bp,168bp) .. controls (262.19bp,158.26bp) and (286.63bp,151.88bp)  .. (u1*u0);
  \draw (251.36bp,178bp) node {$0$};
  \draw [black,->] (u0*u1*u2*u0) ..controls (75.947bp,269.17bp) and (70.733bp,262.83bp)  .. (68.359bp,256bp) .. controls (65.223bp,246.99bp) and (66.522bp,236.41bp)  .. (u0*u1*u2);
  \draw (77.359bp,246bp) node {$2$};
  \draw [black,->] (u1*u0) ..controls (342.16bp,126.94bp) and (361.5bp,104.65bp)  .. (u1);
  \draw (372.36bp,110bp) node {$2$};
  \draw [black,->] (u0*u2*u1) ..controls (549.49bp,196.54bp) and (522.98bp,178.91bp)  .. (498.36bp,168bp) .. controls (488.13bp,163.47bp) and (461.18bp,155.68bp)  .. (u2*u1);
  \draw (543.36bp,178bp) node {$1$};
  \draw [black,->] (u2*u1*u0) ..controls (342.34bp,202.2bp) and (333.24bp,196.4bp)  .. (328.36bp,188bp) .. controls (323.55bp,179.73bp) and (323.21bp,169.03bp)  .. (u1*u0);
  \draw (337.36bp,178bp) node {$0$};
  \draw [black,->] (u2*u1*u0) ..controls (361.97bp,197.09bp) and (354.52bp,180.83bp)  .. (346.36bp,168bp) .. controls (344.18bp,164.57bp) and (341.6bp,161.05bp)  .. (u1*u0);
  \draw (366.36bp,178bp) node {$3$};
  \draw [black,->] (u1*u2*u0) ..controls (165.75bp,201.24bp) and (171.25bp,194.46bp)  .. (175.36bp,188bp) .. controls (181.16bp,178.88bp) and (186.6bp,168.03bp)  .. (u2*u0);
  \draw (198.86bp,178bp) node {$-1$};
  \draw [black,->] (u1*u2*u0) ..controls (150.66bp,196.79bp) and (146.32bp,179.85bp)  .. (153.36bp,168bp) .. controls (158.03bp,160.13bp) and (166.3bp,154.67bp)  .. (u2*u0);
  \draw (162.36bp,178bp) node {$2$};
  \draw [black,->] (u0*u1) ..controls (444.86bp,127bp) and (418.02bp,103.6bp)  .. (u1);
  \draw (443.36bp,110bp) node {$1$};
  \draw [black,->] (u0*u2*u0) ..controls (24.011bp,195.35bp) and (29.888bp,174.32bp)  .. (u0*u2);
  \draw (39.359bp,178bp) node {$0$};
  \draw [black,->] (u1*u2*u1*u0) ..controls (252.78bp,262.89bp) and (271.74bp,241.14bp)  .. (280.36bp,236bp) .. controls (298.67bp,225.09bp) and (322.05bp,219.02bp)  .. (u2*u1*u0);
  \draw (292.86bp,246bp) node {$-1$};
  \draw [black,->] (u1*u0) ..controls (318bp,132.89bp) and (312.29bp,126.06bp)  .. (307.36bp,120bp) .. controls (299.22bp,110bp) and (290.16bp,98.604bp)  .. (u0);
  \draw (316.36bp,110bp) node {$2$};
  \draw [black,->] (u0*u1*u2*u0) ..controls (179.19bp,274.39bp) and (374.92bp,261.94bp)  .. (387.36bp,256bp) .. controls (398.42bp,250.72bp) and (395.94bp,242.46bp)  .. (406.36bp,236bp) .. controls (421.53bp,226.6bp) and (440.65bp,220.64bp)  .. (u0*u1*u0);
  \draw (415.36bp,246bp) node {$1$};
  \draw [black,->] (u0*u1*u2*u0) ..controls (56.765bp,270.21bp) and (45.283bp,264.41bp)  .. (37.359bp,256bp) .. controls (30.034bp,248.23bp) and (25.515bp,237.07bp)  .. (u0*u2*u0);
  \draw (49.859bp,246bp) node {$-1$};
  \draw [black,->] (u0*u1*u2*u0) ..controls (42.039bp,272.73bp) and (10.711bp,265.75bp)  .. (3.3585bp,256bp) .. controls (-3.2495bp,247.24bp) and (1.5614bp,235.48bp)  .. (u0*u2*u0);
  \draw (12.359bp,246bp) node {$2$};
  \draw [black,->] (u0) ..controls (272.36bp,59.313bp) and (272.36bp,38.603bp)  .. (1);
  \draw (281.36bp,42bp) node {$1$};
  \draw [black,->] (u0*u2) ..controls (58.353bp,127bp) and (85.886bp,103.6bp)  .. (u2);
  \draw (96.359bp,110bp) node {$1$};
  \draw [black,->] (u0*u1*u2) ..controls (71.304bp,197.1bp) and (69.363bp,180.38bp)  .. (75.359bp,168bp) .. controls (77.5bp,163.58bp) and (80.883bp,159.69bp)  .. (u1*u2);
  \draw (84.359bp,178bp) node {$1$};
  \draw [black,->] (u1*u2*u0) ..controls (141.02bp,201.44bp) and (133.03bp,195.01bp)  .. (127.36bp,188bp) .. controls (120.42bp,179.43bp) and (114.67bp,168.43bp)  .. (u1*u2);
  \draw (136.36bp,178bp) node {$2$};
  \draw [black,->] (u2*u0) ..controls (216.1bp,127bp) and (241.92bp,103.6bp)  .. (u0);
  \draw (252.36bp,110bp) node {$0$};
  \draw [black,->] (u0*u1*u2*u0) ..controls (91.912bp,264.62bp) and (98.414bp,247.04bp)  .. (109.36bp,236bp) .. controls (114.86bp,230.45bp) and (121.91bp,225.94bp)  .. (u1*u2*u0);
  \draw (121.86bp,246bp) node {$-2$};
  \draw [black,->] (u0*u1*u2*u0) ..controls (115.08bp,269.95bp) and (126.36bp,264.07bp)  .. (134.36bp,256bp) .. controls (142.22bp,248.07bp) and (147.8bp,236.75bp)  .. (u1*u2*u0);
  \draw (156.36bp,246bp) node {$1$};
  \draw [black,->] (u0*u2*u1*u0) ..controls (506.45bp,270.58bp) and (522.15bp,264.63bp)  .. (534.36bp,256bp) .. controls (545.48bp,248.14bp) and (555.27bp,236.15bp)  .. (u0*u2*u1);
  \draw (563.36bp,246bp) node {$4$};
  \draw [black,->] (u0*u1*u0) ..controls (479.46bp,195.35bp) and (473.28bp,174.32bp)  .. (u0*u1);
  \draw (485.36bp,178bp) node {$2$};
  \draw [black,->] (u2*u1*u0) ..controls (380.83bp,195.04bp) and (396.98bp,173.07bp)  .. (u2*u1);
  \draw (407.36bp,178bp) node {$3$};
  \draw [black,->] (u1*u2*u1*u0) ..controls (238.36bp,263.45bp) and (238.36bp,242.73bp)  .. (u1*u2*u1);
  \draw (247.36bp,246bp) node {$3$};
  \draw [black,->] (u0*u2*u1*u0) ..controls (420.62bp,275.1bp) and (376.94bp,268.81bp)  .. (341.36bp,256bp) .. controls (324.14bp,249.8bp) and (322.91bp,241.21bp)  .. (305.36bp,236bp) .. controls (195.16bp,203.31bp) and (161.17bp,234.08bp)  .. (47.359bp,218bp) .. controls (47.257bp,217.99bp) and (47.155bp,217.97bp)  .. (u0*u2*u0);
  \draw (350.36bp,246bp) node {$4$};
  \draw [black,->] (u2*u1*u0) ..controls (336.51bp,201.85bp) and (317.49bp,195.2bp)  .. (301.36bp,188bp) .. controls (284.29bp,180.38bp) and (281.43bp,175.62bp)  .. (264.36bp,168bp) .. controls (249.94bp,161.56bp) and (233.21bp,155.57bp)  .. (u2*u0);
  \draw (310.36bp,178bp) node {$3$};
  \draw [black,->] (u0*u2*u1*u0) ..controls (458.79bp,264.46bp) and (444.28bp,246.77bp)  .. (428.36bp,236bp) .. controls (418.57bp,229.38bp) and (406.68bp,224.14bp)  .. (u2*u1*u0);
  \draw (459.36bp,246bp) node {$1$};
  \draw [black,->] (u0*u2*u1*u0) ..controls (420.37bp,274.59bp) and (378.99bp,267.91bp)  .. (369.36bp,256bp) .. controls (363.14bp,248.31bp) and (362.95bp,237.16bp)  .. (u2*u1*u0);
  \draw (378.36bp,246bp) node {$4$};
  \draw [black,->] (u1*u2*u1*u0) ..controls (217.86bp,263bp) and (189.64bp,239.6bp)  .. (u1*u2*u0);
  \draw (215.36bp,246bp) node {$3$};
  \draw [black,->] (u0*u2*u1*u0) ..controls (496.78bp,270.06bp) and (505.65bp,264.32bp)  .. (510.36bp,256bp) .. controls (514.74bp,248.26bp) and (513.92bp,244.14bp)  .. (510.36bp,236bp) .. controls (508.5bp,231.76bp) and (505.48bp,227.93bp)  .. (u0*u1*u0);
  \draw (521.36bp,246bp) node {$1$};
  \draw [black,->] (u0*u2*u1*u0) ..controls (473.77bp,263.45bp) and (478.03bp,242.73bp)  .. (u0*u1*u0);
  \draw (487.36bp,246bp) node {$4$};
\end{tikzpicture}
\end{center}
\caption{$\downgraph$ for $k=2$ truncated at the affine Grassmannian elements of
length $4$.}
\label{fig:downgraph}
\end{figure}

A \emph{strong strip} of length $i$ from $w$ to $v$, denoted by $w
\strongstrip v$, is a path 
\begin{align*}
            w
            \stackrel{\ell_1}\longrightarrow w_1
            \stackrel{\ell_2}\longrightarrow
            \cdots
            \stackrel{\ell_i}\longrightarrow w_i = v
\end{align*}
of length $i$ in $\downgraph$ with decreasing edge labels:
            $\ell_1 > \cdots > \ell_i$.
For non-negative integers $i$, define $D_i: \AA \to \AA$ as
\[
 D_i(\uu_w)
 = \sum_{
            w \strongstrip v
            \atop
            \size(w \strongstrip v) = i
        } \uu_{v}
\]
where the sum ranges over all strong strips of length $i$ that begin at
$w$. In particular, the coefficient of $\uu_{v}$ in $D_i(\uu_w)$ is the
number of strong strips of length $i$ that begin at $w$ and end at $v$.

\begin{Example}
With $k=2$, using the graph from Figure \ref{fig:downgraph}, one can verify that:
\begin{align*}
    D_1(\uu_0\uu_1\uu_2\uu_0) 
    &= 2\uu_0\uu_2\uu_0 + \uu_0\uu_1\uu_2 + 2\uu_1\uu_2\uu_0 + \uu_0\uu_1\uu_0;
    \\
    D_2(\uu_0\uu_1\uu_2\uu_0) 
    &= \uu_0\uu_2 + \uu_1\uu_2 + \uu_2\uu_0 + \uu_1\uu_0.
\end{align*}
\end{Example}

More generally, we define an operator $D_J$ for any composition $J$ of positive
integers; the operator $D_i$ defined above is $D_J$ for the composition $J=[i]$.
We need some additional notation. 
The \emph{ascent composition} of a sequence $\ell_1, \ell_2, \dots, \ell_m$ is
the composition
$[
i_1,
i_2 - i_{1},
\ldots,
i_j - i_{j-1},
m - i_j
]$,
where $i_1 < i_2 < \dots < i_j$ are the ascents of the sequence; that is, the elements in $\{1, \dots, m-1\}$ such
that $\ell_{i_a} < \ell_{i_a+1}$. For example, the ascent composition of the sequence $3,2,0,3,4,1$ is $[3,1,2]$ since the ascents are in positions $3$ and $4$.

If
$ w_0 \stackrel{\ell_1}\longrightarrow \cdots \stackrel{\ell_m}\longrightarrow w_m $
is a path in $\downgraph$, then we let 
$\ascomp(w_0 \stackrel{\ell_1}\longrightarrow \cdots \stackrel{\ell_m}\longrightarrow w_m)$
denote the ascent composition of the sequence of labels $\ell_1, \dots, \ell_m$.
It is a composition of the length of the path.

For a composition $J = [j_1, j_2, \dots, j_l]$ of positive integers, define
\[
 D_J(\uu_w)
 = \sum_{
         \ascomp\left(
            w 
            \stackrel{\ell_1}\longrightarrow w_1
            \stackrel{\ell_2}\longrightarrow
            \cdots
            \stackrel{\ell_m}\longrightarrow w_m
        \right)
        = J
        } \uu_{w_m}
\]
where the sum ranges over all paths in $\downgraph$ of length $m = j_1 + \dots
+ j_l$ beginning at $w$ whose sequence of labels has ascent composition $J$.

\begin{Example}
With $k=2$ one can verify using Figure \ref{fig:downgraph} that:
    \begin{align*}
    D_{[3]}(\uu_1\uu_2\uu_1\uu_0) &= \uu_2 + \uu_0, \\
    D_{[2,1]}(\uu_1\uu_2\uu_1\uu_0) & = \uu_2 + 2\uu_0 + \uu_1, \\
    D_{[1,2]}(\uu_1\uu_2\uu_1\uu_0) & = \uu_2 + 2\uu_0 + \uu_1, \\
    D_{[1,1,1]}(\uu_1\uu_2\uu_1\uu_0) & = \uu_0 + \uu_1.
    \end{align*}
\end{Example}

For two compositions $I = [i_1, \dots, i_r]$ and $J = [j_1, \dots, j_s]$,
let
\begin{align*}
I \boxplus J &= [i_1, \dots, i_{r-1}, i_r+j_1, , j_2, \dots, j_s] \\
I \boxdot J &= [i_1, \dots, i_r, j_1, \dots, j_s]
\end{align*}

\begin{Proposition}
    If $I$ and $J$ are compositions, then
    \begin{gather*}
        D_I \circ D_J = D_{I \boxplus J} + D_{I \boxdot J}.
    \end{gather*}
\end{Proposition}
\begin{proof}
    If $w \to \dots \to v$ and $v \to \dots \to u$ are two paths in
    $\downgraph$ with ascent compositions $J$ and $I$, respectively,
    then the path $w \to \dots \to v \to \dots \to u$ has
    ascent composition either $I \boxdot J$ or $I \boxplus J$.
\end{proof}

\begin{Corollary}
    Suppose $I = [i_1, \dots, i_r]$ is a composition. Then
    \begin{align*}
        \sum_{J \preceq I} D_J = D_{i_1} \circ \cdots \circ D_{i_r}
    \end{align*}
    where $\preceq$ denotes reverse refinement order on
    compositions\footnote{With respect to this order, the composition
    $[1,1,\dots,1]$ is the maximal element and the composition $[n]$ is
    the minimal element.}.
\end{Corollary}
\begin{proof}
    Proceed by induction on $r$. This is trivially true for $r=1$.
    Suppose the result holds for compositions of length less than $r$.
    Then
    \begin{align*}
        D_{i_1} \circ \cdots \circ D_{i_r}
        =
        D_{i_1} \left( 
            \sum_{J' \preceq [i_2, \dots, i_r]} D_{J'}
            \right)
        =
        \sum_{J' \preceq [i_2, \dots, i_r]}
            \left(
                D_{[i_1, j_1, \dots, j_s]} + D_{[i_1 + j_1, \dots, j_s]}
            \right)
    \end{align*}
    which is $\sum_{J \preceq I} D_J$ since the first part of a
    composition $J$ that satisfies $J \preceq I$ is either 
    $i_1$ or $i_1 + (i_2 + \dots + i_l)$ for some $l \geq 2$.
\end{proof}

\section{Properites of the operators} \label{sec:operators}

In this section we develop properties of the operators $U_j$ and $D_i$.

\subsection{Extensions of linear endomorphisms of $\BB$ to $\AA$}
\label{ss:extensions}
Since $W^0$ is a set of coset representatives of $W_0$ in $W$, every element
$w$ of $W$ factors uniquely as $w = w^{(0)} w_{(0)}$ with $w^{(0)} \in W^0$ and
$w_{(0)} \in W_0$. We call this the \emph{$0$-Grassmannian factorization} of
$w$. Since the elements of $W^0$ are in bijection with $k$-bounded partitions,
we can write this factorization as $w = w_\lambda w_{(0)}$, and we let
\begin{gather*}
    \bb_w = \nckschur_\lambda \uu_{w_{(0)}}.
\end{gather*}

\begin{Proposition}
    \label{prop:newAbasis}
    The set $\{\bb_w : w \in W\}$ is a basis of $\AA$.
\end{Proposition}

\begin{proof}
    We will define a total order on the elements of $W$ in such a way that the
    leading term of $\bb_w$ is $\uu_w$. Then, with respect to this ordering,
    the transition matrix from $\{\bb_w\}$ to $\{\uu_w\}$ is uni-triangular,
    from which the result follows.

    Informally, we need an order in which $v$ precedes $u$ whenever $\ell(u) >
    \ell(v)$ or the ``Grassmannian part'' of $u$ is bigger than that of $v$.
    Define $v$ to precede $u$ if:
            %$\ell(u) > \ell(v)$; or
            %$\ell(u) = \ell(v)$ and $\ell(u^{(0)}) > \ell(v^{(0)})$; or
            %$\ell(u) = \ell(v)$, $\ell(u^{(0)}) = \ell(v^{(0)})$
            %and $\mu >_{\text{lex}} \nu$, where
            %$\mu$ and $\nu$ are the $k$-bounded partitions corresponding to
            %$u^{(0)}$ and $v^{(0)}$, respectively.
            $\ell(u) > \ell(v)$; or
            $\ell(u) = \ell(v)$ and $\ell(u^{(0)}) > \ell(v^{(0)})$.
    Note that this is only a partial order, but any linear extension of this
    partial order will do the trick.

    First we argue that the leading term of $\bb_{w_\lambda} = \nckschur_\lambda$ is
    $\uu_{w_\lambda}$. Indeed, $\nckschur_\lambda$ expanded in the basis
    $\{\uu_v\}$ is a linear combination of terms $\uu_v$ with the $v$ all of
    the same length $|\lambda|$, and it contains exactly one term $\uu_w$ with
    $w \in W^0$, namely $w_\lambda$ (see \S\ref{ss:noncommkschurs}).

    Next, we prove that the leading term of $\bb_w$ is $\uu_w$.
    If $\uu_v$ appears in $\bb_w = \nckschur_\lambda \uu_{w_{(0)}}$ with nonzero
    coefficient, then 
        $\uu_v = \uu_{\tilde v} \uu_{w_{(0)}}$
    with $\uu_{\tilde v}$ appearing in $\nckschur_\lambda$.
    It follows that
        $v = \tilde vw_{(0)}$ with $\ell(v) = \ell(\tilde v) + \ell(w_{(0)})$
    and that
        $v^{(0)} = {\tilde v}^{(0)}$.
    Hence, to compare terms $\uu_v$ and $\uu_u$ of $\bb_w$, it suffices to
    compare the corresponding terms $\uu_{\tilde v}$ and $\uu_{\tilde u}$ of
    $\nckschur_\lambda$. So, the leading term of $\bb_w$ is the leading term
    of $\nckschur_\lambda$ times $\uu_{w_{(0)}}$, which is precisely $\uu_{w_\lambda} \uu_{w_{(0)}} = \uu_w$.
\end{proof}

\begin{Corollary}
    \label{coro:newAbasis}
    The set $\{\hh_\lambda \uu_{w_{(0)}} : w = w_\lambda w_{(0)} \in
    W\}$ is a basis of $\AA$.
\end{Corollary}
\begin{proof}
    Follows from Proposition \ref{prop:newAbasis} and the fact that
    $\{\hh_\lambda\}$ is a basis of $\BB$.
\end{proof}

The above results allow us to extend linear endomorphisms of $\BB$ to linear
endomorphisms of $\AA$. Let $f: \BB \to \BB$ be a linear transformation of
$\BB$. Then we get a linear transformation $\widehat{f}: \AA \to \AA$ by
defining $\widehat{f}$ on the basis $\{\bb_w\}$ by
\begin{gather*}
    \widehat{f}\left(\bb_w\right) = \widehat{f}\left(\nckschur_\lambda\uu_{w_{(0)}}\right)
    = f\left(\nckschur_\lambda\right) \uu_{w_{(0)}},
\end{gather*}
where $w = w_\lambda w_{(0)}$ is the $0$-Grassmannian factorization of $w$.

\subsection{Commutation relation}
We prove a commutation relation between the operators $U_j$ and $D_i$. This
relation will allow us to bootstrap properties of $D_1$ and $U_j$ to every
operator $D_i$ via an inductive argument.

\begin{Proposition}[Commutation Relation]
    \begin{gather*}
        D_i \circ U_j = \sum_{e \geq 0} U_{j-e} \circ D_{i-e}
    \end{gather*}
\end{Proposition}
\begin{proof}
First note that the right hand side is a finite sum. The coefficient of $\uu_v$ in
\begin{align*}
  \left(D_i \circ U_j\right)(\uu_u)
= \sum_{ u \weakstrip w \atop \size(u \weakstrip w) = j }
  \sum_{ w \strongstrip v \atop \size(w \strongstrip v) = i } \uu_{v}
\end{align*}
is the number of tuples $(u \weakstrip w, w \strongstrip v)$
consisting of
a weak strip $u \weakstrip w$ of length $j$ and
a strong strip $w \strongstrip v$ of length $i$.
The coefficient of $\uu_v$ in
\begin{align*}
  \left(\sum_{e \geq 0} U_{j-e} \circ D_{i-e}\right)(\uu_u)
= \sum_{e \geq 0}
  \sum_{ u \strongstrip x \atop \size(u \strongstrip x) = i-e }
  \sum_{ x \weakstrip v \atop \size(x \weakstrip v) = j-e } \uu_{v}
\end{align*}
is the number of triples $(e, u \strongstrip x, x \weakstrip v)$
consisting of
a nonnegative integer $e$,
a strong strip $u \strongstrip x$ of length $i-e$ and
a weak strip $x \weakstrip v$ of length $j-e$.

By \cite[Proposition~4.1]{LLMS10}, these two numbers are the same. Indeed,
that proposition establishes a bijection between the sets:
\begin{align*}
    &\left\{
    (W', S') :
    \begin{array}{l}
        \text{$W'$ is a weak strip beginning at $u$,} \\
        \text{$S'$ is a strong strip ending at $v$}, \\
        \text{with $W'$ ending where $S'$ begins.}
    \end{array}
    \right\}
    \\
    &\qquad\qquad\qquad\qquad
    \longleftrightarrow
    \left\{
    (W, S, e) :
    \begin{array}{l}
        \text{$W$ is a weak strip ending at $v$,} \\
        \text{$S$ is a strong strip beginning at $u$,} \\
        \text{$e\geq0$ satisfies $\size(W)+e \leq k$,} \\
        \text{with $S$ ending where $W$ begins.}
    \end{array}
    \right\}
\end{align*}
such that
\begin{align*}
\size(S) &= \size(S') - e \\
\size(W) &= \size(W') - e. \qedhere
\end{align*}
\end{proof}

\begin{Corollary}[Bracket]
    \begin{align*}
        D_i \circ U_j - U_j \circ D_i = D_{i-1} \circ U_{j-1}
    \end{align*}
\end{Corollary}
\begin{proof}
    \begin{gather*}
        D_i \circ U_j 
        = \sum_{e \geq 0} U_{j-e} \circ D_{i-e}
        = U_{j} \circ D_{i} + \sum_{e \geq 1} U_{j-e} \circ D_{i-e}
        = U_{j} \circ D_{i} + D_{i-1} \circ U_{i-1}
        \qedhere
    \end{gather*}
\end{proof}

\subsection{$D_i$ stabilizes $\BB$}
We use the commutation relation of the previous section to prove that $D_i(\BB)
\subseteq \BB$. First we determine the image of $\hh_r$ under $D_i$.

\begin{Lemma}
    \label{Dh_r}
    For $r \leq k$ and all $i$,
    $$D_i(\hh_r) = \hh_{r-i}.$$
\end{Lemma}
\begin{proof}
    If $i > r$, then $D_i(\hh_r) = 0$ since there are no
    strong strips of size $i$ beginning at elements $w$ of length $r$.
    Also, by definition, $\hh_{r-i} = 0$. So suppose that $i \leq r$.
    Proceed by induction on $i$. If $i = 1$, then
    \begin{align*}
        D_1(\hh_r) 
        = (D_1 \circ U_r)(1_\AA) 
        = (U_r \circ D_1)(1_\AA) + (D_0 \circ U_{r-1})(1_\AA) 
        = 0_\AA + \hh_{r-1}.
    \end{align*}
    Suppose the result holds for $i-1$. Then
    \begin{align*}
        D_i(\hh_r) 
        &= (D_i \circ U_r)(1_\AA)
        = (U_r \circ D_i)(1_\AA) + (D_{i-1} \circ U_{r-1})(1_\AA) \\
        &= 0_\AA + D_{i-1}(\hh_{r-1}) = \hh_{r-i}.
        \qedhere
    \end{align*}
\end{proof}

\begin{Theorem}
    Let $J$ be a composition. Then $D_J$ stabilizes $\BB$; that is,
        $$D_J(\BB) \subseteq \BB.$$
\end{Theorem}
\begin{proof}
It suffices to prove this for the operators $D_i$ since $D_J$ is a linear
combination of compositions of these operators.
Since $\BB$ is spanned by the products
$\hh_{j_1} \hh_{j_2} \cdots \hh_{j_l}$,
it suffices to show that
$D_i(\hh_{j_1} \hh_{j_2} \cdots \hh_{j_l}) \in \BB$.
Proceed by induction on $i$ and $l$.
If $l=1$, then by Lemma \ref{Dh_r}, $D_i(\hh_{j_1}) = \hh_{j_1-i} \in
\BB$. For $i=1$ this was proved in \cite[Theorem~3.9]{BSS11}.
If $l > 1$, then
\begin{align*}
D_i(\hh_{j_1} \hh_{j_2} \cdots \hh_{j_l})
&= \left(D_i \circ U_{j_1}\right)(\hh_{j_2} \cdots \hh_{j_l}) \\
&= \left(U_{j_1} \circ D_i\right)(\hh_{j_2} \cdots \hh_{j_l})
+ \left(D_{i-1} \circ U_{j_1-1}\right)(\hh_{j_2} \cdots \hh_{j_l}) \\
&= \hh_{j_1} D_i(\hh_{j_2} \cdots \hh_{j_l})
+ D_{i-1}(\hh_{j_1-1} \hh_{j_2} \cdots \hh_{j_l}) \in \BB.
\qedhere
\end{align*}
\end{proof}

Since the noncommutative $k$-Schur functions form a basis of $\BB$, it is
natural to ask for the expansion of $D_i(\nckschur_\lambda)$ in terms of
noncommutative $k$-Schur functions. We obtain the following combinatorial description in terms
of strong strips. Recall that $w_\lambda$ denotes the $0$-Grassmannian element
corresponding to the $k$-bounded partition $\lambda$ under the bijection
between $\kBounded$ and $W^0$.

\begin{Theorem}
    \label{PieriRulePerp}
    $$D_i\left(\nckschur_\lambda\right) 
    = \sum_{\size(w_\lambda \strongstrip w_\mu)=i} \nckschur_\mu.$$
\end{Theorem}
\begin{proof}
    Since $D_i(\nckschur_\lambda) \in \BB$, to compute its expansion in terms
    of $k$-Schur functions, it suffices to compute the coefficient of
    $\uu_{w}$ for $0$-Grassmannian elements $w$ (see \S\ref{ss:noncommkschurs}). This
    is the number of strong strips $v \strongstrip w$ of length $i$ with
    $\uu_v$ appearing as a term in $\nckschur_\lambda$. But a strong strip that
    ends at a $0$-Grassmannian element necessarily begins at a $0$-Grassmannian
    element \cite[Proposition~2.6]{LLMS10}, and there is a unique term $\uu_v$
    appearing in $\nckschur_\lambda$ with $v$ a $0$-Grassmannian element, namely
    $\uu_{w_\lambda}$.
\end{proof}

\subsection{Restriction to $\BB$}

We prove that $D_J$ is determined by its restriction to $\BB$ and we
identify this restriction as a linear operator adjoint to multiplication by a
symmetric function with respect to the pairing on $\Lambda_{(k)} \times
\Lambda^{(k)}$.

\begin{Theorem}
    \label{sym-module-morphism}
    Suppose $w \in W$ and $v \in W_0$. Then
    \begin{gather*}
        U_j(\uu_w \uu_v) = U_j(\uu_w) \uu_v \\
        D_i(\uu_w \uu_v) = D_i(\uu_w) \uu_v
    \end{gather*}
    Consequently, $U_j$ and $D_i$ are completely determined by their
    restriction to $\BB$.
\end{Theorem}
\begin{proof}
Since $U_j$ is left-multiplication by $\hh_j$, associativity
implies that $U_j(\uu_w \uu_v) = U_j(\uu_w) \uu_v$, establishing
the first equality.
By Corollary \ref{coro:newAbasis}, it suffices to show that
$D_i(\hh_{j_1} \hh_{j_2} \cdots \hh_{j_l} \uu_v) =
D_i(\hh_{j_1} \hh_{j_2} \cdots \hh_{j_l}) \uu_v$.
Proceed by induction. The case $i=1$ was proved in \cite[Theorem~3.10]{BSS11}.
Suppose the result holds for $D_{i-1}$. We prove the result also holds
for $D_i$ by induction on $l$. If $l=1$, then
\begin{align*}
    D_i(\hh_{j_1} \uu_v)
    = (D_i \circ U_{j_1})(\uu_v)
    = (U_{j_1} \circ D_{i})(\uu_v) + (D_{i-1} \circ U_{j_1-1})(\uu_v).
\end{align*}
Note that $D_i(\uu_v) = 0$ because there is no strong strip starting
from $v \in W_0$. And since the result holds for $D_{i-1}$, we have
\begin{align*}
    (D_{i-1} \circ U_{j_1-1})(\uu_v)
    = D_{i-1}(\hh_{j_1-1}\uu_v)
    = D_{i-1}(\hh_{j_1-1})\uu_v
    %= \hh_{(j_1-1)-(i-1)}\uu_v
    = D_i(\hh_{j_1})\uu_v.
\end{align*}
For $l > 1$, use the identity
$D_i \circ U_{j_1} = U_{j_1} \circ D_i + D_{i-1} \circ U_{j_1-1}$
to write
\begin{align*}
    D_i(\hh_{j_1} \hh_{j_2} \cdots \hh_{j_l} \uu_v)
    &= (U_{j_1} \circ D_i)(\hh_{j_2} \cdots \hh_{j_l} \uu_v)
    + D_{i-1} (\hh_{j_1-1}\hh_{j_2} \cdots \hh_{j_l} \uu_v).
\end{align*}
Since the product $\hh_{j_2} \cdots \hh_{j_l}$ involves less than $l$
terms, by induction we have that
\begin{align*}
    (U_{j_1} \circ D_i)(\hh_{j_2} \cdots \hh_{j_l} \uu_v)
    = U_{j_1} \left( D_i(\hh_{j_2} \cdots \hh_{j_l}) \uu_v \right)
    = (U_{j_1} \circ D_i)(\hh_{j_2} \cdots \hh_{j_l}) \uu_v.
\end{align*}
Since the result holds for $D_{i-1}$, we have that
\begin{align*}
    D_{i-1} (\hh_{j_1-1}\hh_{j_2} \cdots \hh_{j_l} \uu_v)
    = D_{i-1} (\hh_{j_1-1}\hh_{j_2} \cdots \hh_{j_l}) \uu_v.
\end{align*}
Hence,
    $D_i(\hh_{j_1} \hh_{j_2} \cdots \hh_{j_l} \uu_v)
    = D_i(\hh_{j_1} \hh_{j_2} \cdots \hh_{j_l}) \uu_v$,
as desired.
\end{proof}

We next identify the restriction of $D_J$ to $\BB$. For a composition $J$, let
$s_J$ denote the \emph{ribbon Schur function} indexed by $J$ (for a good introduction to ribbon Schur functions, see for instance \cite{BTW06}) and
let $\ribbonfcn{J}$ denote its image in $\Lambda^{(k)}$.
Recall that
    $\ribbonfcn{J}^\perp:\Lambda_{(k)} \to \Lambda_{(k)}$
is the linear operator adjoint to multiplication by $\ribbonfcn{J}$ in 
$\Lambda^{(k)}$. We also denote the corresponding linear operator on $\BB$ by
$\ribbonfcn{J}^\perp$.
\begin{Theorem}
    \label{restriction}
    The restriction of $D_J$ to $\BB$ is $\ribbonfcn{J}^\perp$.
    Consequently, $D_J$ is the extension to $\AA$, as defined in
    \S\ref{ss:extensions}, of $\ribbonfcn{J}^\perp:\BB \to \BB$.
    %$D_J = \widehat{\ribbonfcn{J}}^\perp$.
\end{Theorem}
\begin{proof}
    In the following, let $D_J(\kschur_\lambda)$ denote the image
    $D_J(\nckschur_\lambda)$ under the isomorphism $\BB \to \Lambda_{(k)}$.
    We will prove, for all $\kschur_\lambda$ and $\dualkschur_\mu$,
    $$\Langle D_J\left( \kschur_\lambda \right), \dualkschur_{\mu} \Rangle 
    = \Langle \kschur_\lambda, \ribbonfcn{J}\dualkschur_{\mu} \Rangle.$$
    %where $\langle \cdot, \cdot \rangle: \Lambda_{(k+1)} \times
    %\Lambda^{(k+1)} \to \mathbb Z$ is the inner product induced by
    %the Hall inner product on $\Lambda\times\Lambda$.
    Proceed by induction on the length of $J = [j_1,j_2,\dots,j_l]$.
    Suppose $l = 1$. Then it suffices to prove that
    \begin{gather*}
        \Langle D_j\left( \kschur_\lambda \right), \dualkschur_\mu \Rangle 
        = \Langle \kschur_\lambda, \overline{h_j}\dualkschur_\mu \Rangle.
    \end{gather*}
    But this follows immediately from Theorem \ref{PieriRulePerp} and
    the Pieri rule: $\overline{h_j} \dualkschur_{w_\mu} = \sum
    \dualkschur_{w_\lambda}$ with the sum running over all strong strips
    ${w_\lambda} \strongstrip w_\mu$ of size $j$ (see \cite[Theorem 4.13]{LLMS10}).

    Now suppose the result holds for compositions of length less than
    $l$. Let $J = [j_1, j_2, \dots, j_l]$. Observe that
    \begin{align*}
        D_J = D_{j_1} \circ D_{[j_2, \dots, j_l]} - D_{[j_1 + j_2, \dots, j_l]},
    \end{align*}
    so by induction and the product rule for ribbon Schur functions
    \cite[\S 169]{Mac16},
    \begin{align*}
        D_J 
        &=
        \ribbonfcn{j_1}^\perp \circ \ribbonfcn{[j_2, \dots, j_l]}^\perp - 
        \ribbonfcn{[j_1 + j_2, \dots, j_l]}^\perp, \\
        &=
        \overline{s_{j_1}s_{[j_2, \dots, j_l]}
        - s_{[j_1 + j_2, \dots, j_l]}}^\perp
        = \ribbonfcn{J}^\perp.
        \qedhere
    \end{align*}
\end{proof}

\begin{Theorem}\label{cor:commute}
    The operators $D_J$ and $D_K$ commute.
\end{Theorem}
\begin{proof}
    $\AA$ is spanned by elements of the form $\bb \uu_w$ with $\bb \in \BB$ and
    $w \in W_0$ (Proposition \ref{prop:newAbasis}), so it suffices to prove
    this for these elements. Combining Theorems
    \ref{sym-module-morphism} and \ref{restriction}, we have
    \begin{align*}
    &\left(D_K \circ D_J\right)\left(\bb\uu_w\right)
     = \left(D_K \circ D_J\right)\left(\bb\right)\uu_w
     = \left(\ribbonfcn{K}^\perp \circ
     \ribbonfcn{J}^\perp\right)\left(\bb\right)\uu_w
    \\
    &= \left(\ribbonfcn{J}^\perp \circ
    \ribbonfcn{K}^\perp\right)\left(\bb\right)\uu_w
     = \left(D_J \circ D_K\right)\left(\bb\right)\uu_w
     = \left(D_J \circ D_K\right)\left(\bb\uu_w\right),
    \end{align*}
    where the third equality comes from the commutation of symmetric functions.
\end{proof}

\section{Strong Schur functions}\label{sec:strong}

Lam, Lapointe, Morse, and Shimozono \cite{LLMS10} generalized the $k$-Schur
functions to a larger set of functions called the strong Schur functions. We
use the properties of the operators developed in the previous section to prove
a series of their conjectures \cite[Conjecture~4.18]{LLMS10} regarding these
functions.

\subsection{Strong Schur functions are symmetric functions}

For $u, v \in W$, define the \emph{strong Schur function}
\begin{align*}
    \Strong_{u/v}
    &= \sum_{ u \to \cdots \to v \in \downgraph }
        F_{\ascomp(u \to \dots \to v)},
\end{align*}
where $F_J$ denotes the fundamental quasi-symmetric function indexed by the
composition $J$. In \cite{LLMS10}, it was shown that $\Strong_{u/\id}$ is a
symmetric function; and that when $u$ is $0$-Grassmannian, it is a $k$-Schur
function.

\begin{Remark}
    \label{remark:strong tableaux}
    The definition given here is a reformulation of that in \cite{LLMS10}.
    They defined $\Strong_{u/v}$ as the generating function of ``strong
    tableaux''; the above definition is obtained from theirs by lumping
    together tableaux of the same ``weight'', yielding the expansion in
    terms of monomial quasisymmetric functions below.
\end{Remark}

\begin{Theorem}[{\cite[Conjecture~4.18(1)]{LLMS10}}]\label{thm:LLMS1}
    $\Strong_{u/v}$ is a symmetric function. Furthermore, it expands
    positively in the monomial basis $m_\lambda$ of $\Lambda$:
    \begin{align*}
        \Strong_{u/v}
        &= \sum_{\lambda}
           \Langle D^{\lambda}(\uu_u), \uu_v \Rangle_\AA
           m_\lambda,
    \end{align*}
    where $D^\lambda = D_{\lambda_1} \circ \dots \circ D_{\lambda_l}$.
\end{Theorem}

\begin{proof}

    The coefficient of the fundamental quasi-symmetric function $F_J$ in
    $\Strong_{u/v}$ is the number of paths in $\downgraph$ from $u$ to $v$
    with ascent composition equal to $J$. This is precisely
    the coefficient of $\uu_v$ in $D_J(\uu_u)$. Hence,
    \begin{align*}
        \Strong_{u/v}
        &= \sum_{J \models \ell(u)-\ell(v)}
            \langle D_J(\uu_u), \uu_v \rangle_\AA
            F_J.
    \end{align*}
    Recall that $F_J = \sum_{I \succeq J} M_I$, where $M_I$
    denotes the monomial quasi-symmetric function indexed by the
    composition $I = [i_1, \dots, i_r]$. Thus,
    \begin{align*}
        \Strong_{u/v}
        &= \sum_{J}
           \langle D_J(\uu_u), \uu_v \rangle_\AA
           \sum_{I \succeq J} M_I \\
        &= \sum_{I}
           \left(
           \sum_{J \preceq I}
           \langle D_J(\uu_u), \uu_v \rangle_\AA
           \right)
           M_I \\
        &= \sum_{I}
           \langle D^I(\uu_u), \uu_v \rangle_\AA
           M_I
    \end{align*}
    where $D^I = D_{i_1} \circ \cdots \circ D_{i_r}$.
    Since the operators $D_{i}$ and $D_{j}$ commute for all $i$ and $j$,
    the operator $D^I$ depends only on the underlying partition
    $\lambda(I)$ of $I$. Hence,
    \begin{align*}
        \Strong_{u/v}
        &= \sum_{\lambda} \sum_{\lambda(I) = \lambda}
           \langle D^I(\uu_u), \uu_v \rangle_\AA
           M_I \\
        &= \sum_{\lambda}
           \langle D^{\lambda}(\uu_u), \uu_v \rangle_\AA
           \sum_{\lambda(I) = \lambda} M_I \\
        &= \sum_{\lambda}
           \langle D^{\lambda}(\uu_u), \uu_v \rangle_\AA
           m_\lambda,
    \end{align*}
    where $m_\lambda$ is the monomial symmetric function. In particular
    $\Strong_{u/v} \in \Lambda$.
\end{proof}

If $u$ and $v$ are $0$-Grassmannian elements, we write $\Strong_{\mu/\nu}$
instead of $\Strong_{u/v}$, where $\mu$ and $\nu$ are the $k$-bounded
partitions corresponding to $u$ and $v$, respectively. 
It follows from \S\ref{ss:noncommkschurs} (as in the proof of Theorem
\ref{PieriRulePerp}) that the coefficient of $\uu_v$ in $D^\lambda(\uu_u)$ is
the coefficient of $\kschur_\nu$ in the expansion of $D^\lambda(\kschur_\mu)$
in terms of $k$-Schur functions. Thus,
\begin{gather*}
    \Langle \strut D^\lambda(\uu_u), \uu_v \Rangle_\AA
    =
    \Langle D^\lambda(\kschur_\mu), \dualkschur_\nu \Rangle
    =
    \Langle \kschur_\mu, \overline{h_\lambda}\dualkschur_\nu \Rangle
\end{gather*}
where the last equality follows from the fact that the restriction of
$D^\lambda$ to $\BB$ is the adjoint to multiplication by $\overline{h_\lambda}$
(Theorem \ref{restriction}).

\begin{Corollary}
    If $u$ and $v$ are $0$-Grassmannian elements corresponding to the
    $k$-bounded partitions $\mu$ and $\nu$, respectively, then
    \begin{align*}
        \Strong_{\mu/\nu}
        &= \sum_{\lambda} 
           \Langle \kschur_\mu, \overline{h_{\lambda}} \dualkschur_\nu \Rangle
           m_\lambda.
    \end{align*}
\end{Corollary}

\subsection{Strong Schur functions belong to $\Lambda_{(k)}$}

Next we verify the second part of Conjecture~4.18 from \cite{LLMS10}.
Recall that for a linear operator $f$ on $\BB$, we denote by $\widehat{f}$ its
extension to $\AA$ as defined in \S\ref{ss:extensions}.

\begin{Theorem}\label{thm:LLMS2}
    Let $u, v \in W$. The strong Schur function
    $\Strong_{u/v}$ lies in $\Lambda_{(k)}$.
    Furthermore, we have the expansion in
    homogeneous symmetric functions:
    \begin{align*}
        \Strong_{u/v}
        =
        \sum_{\lambda\in\kBounded} \Langle \widehat{m_\lambda^\perp} (\uu_u), \uu_v\Rangle_\AA h_\lambda.
    \end{align*}
\end{Theorem}
\begin{proof}
Since the $m_\mu$ form a basis of $\Lambda$, there exist coefficients
$L_{\lambda,\mu}$ for which
$h_\lambda = \sum_{\mu} L_{\lambda,\mu} m_\mu$. Hence,
\begin{align*}
    \Strong_{u/v}
    &=
    \sum_\lambda \Langle D^\lambda(\uu_u), \uu_v\Rangle_\AA m_\lambda
    \\&=
    \sum_\lambda \Langle \widehat{h_\lambda^\perp}(\uu_u), \uu_v\Rangle_\AA m_\lambda
    \\&=
    \sum_{\lambda,\mu} \Langle L_{\lambda,\mu} \widehat{m_\mu^\perp} (\uu_u), \uu_v\Rangle_\AA m_\lambda
    \\&=
    \sum_{\mu} \Langle \widehat{m_\mu^\perp} (\uu_u), \uu_v\Rangle_\AA \sum_{\lambda}L_{\lambda,\mu}m_\lambda
    \\&=
    \sum_{\mu} \Langle \widehat{m_\mu^\perp} (\uu_u), \uu_v\Rangle_\AA h_\mu
\end{align*}
Since $m_\mu^\perp = 0$ for any partition $\mu$ that is not
$k$-bounded, the above summation runs over $k$-bounded partitions.
\end{proof}

\subsection{Expansions of strong Schur functions}

Since $\Strong_{u/v}$ lies in $\Lambda_{(k)}$, it has an expansion in terms of
$k$-Schur functions. The third part of Conjecture~4.18 of \cite{LLMS10} deals
with the coefficients in this expansion.

\begin{Theorem}\label{thm:LLMS3}
    Let $u, v \in W$.
    \begin{align*}
    \Strong_{u/v}
    &=
    \sum_{\lambda\in\kBounded} \Langle \widehat{\dualkschur_\lambda^\perp} (\uu_u), \uu_v\Rangle_\AA \kschur_\lambda
    \end{align*}
\end{Theorem}
\begin{proof}
    By using the expansions
    $h_\lambda = \sum_\tau K^{(k)}_{\tau, \lambda} \kschur_\tau$
    and
    $\dualkschur_\tau = \sum_\lambda K^{(k)}_{\tau, \lambda} m_\lambda$,
\begin{align*}
    \Strong_{u/v}
    &=
    \sum_{\lambda\in\kBounded} \Langle \widehat{m_\lambda^\perp} (\uu_u), \uu_v\Rangle_\AA h_\lambda
    \\&=
    \sum_{\lambda\in\kBounded} \Langle \widehat{m_\lambda^\perp} (\uu_u), \uu_v\Rangle_\AA \sum_{\tau\in\kBounded} K^{(k)}_{\tau,\lambda} \kschur_\tau
    \\&=
    \sum_{\tau\in\kBounded} \Langle \sum_{\lambda\in\kBounded} K^{(k)}_{\tau,\lambda} \widehat{m_\lambda^\perp} (\uu_u), \uu_v\Rangle_\AA \kschur_\tau
    \\&=
    \sum_{\tau\in\kBounded} \Langle \widehat{\dualkschur_\tau^\perp}(\uu_u), \uu_v\Rangle_\AA \kschur_\tau.
    \qedhere
\end{align*}
\end{proof}

If $u$ and $v$ are $0$-Grassmannian, with $u = w_\mu$ and $v =
w_\lambda$, then the coefficient in the above expression reduces
to
\begin{gather*}
\Langle \widehat{\dualkschur_\lambda^\perp}(\uu_u), \uu_v\Rangle_\AA
=
\Langle {\dualkschur_\lambda^\perp}\left( \kschur_\mu \right), \dualkschur_\nu\Rangle
=
\Langle \kschur_\mu, \dualkschur_\lambda \dualkschur_\nu\Rangle.
\end{gather*}
This establishes the third part of Conjecture~4.18 of \cite{LLMS10} for
$0$-Grassmannian elements.

\begin{Corollary}[{\cite[Conjecture~4.18(3)]{LLMS10}}]
    \label{Conj4.18c}
Let $\mu$ and $\nu$ be $k$-bounded partitions.
The coefficient of $\kschur_\lambda$ in $\Strong_{\mu/\nu}$ is the
coefficient of $\dualkschur_\mu$ in $\dualkschur_\lambda \dualkschur_\nu$:
    \begin{align*}
    \Strong_{\mu/\nu}
    &=
    \sum_{\lambda\in\kBounded} \Langle \kschur_\mu, \dualkschur_\lambda \dualkschur_\nu \Rangle \kschur_\lambda.
    \end{align*}
\end{Corollary}

\begin{Corollary}
Let $\mu$ and $\nu$ be $k$-bounded partitions. Then the skew $k$-Schur function is:
    \begin{align*}
    \kschur_{\mu/\nu} := {\dualkschur_\nu^\perp}\left(\kschur_{\mu}\right)
    &=
    \Strong_{\mu/\nu}.
    \end{align*}
\end{Corollary}
\begin{proof}
By Corollary \ref{Conj4.18c}, we have $\Strong_{\mu} = \kschur_\mu$.
Thus, the coefficient of $\kschur_\lambda$ in the left-hand side is
$
    \langle
        {\dualkschur_\nu^\perp}(\kschur_{\mu}),
        \dualkschur_\lambda
    \rangle
    =
    \langle
        \kschur_{\mu},
        \dualkschur_\nu \dualkschur_\lambda
    \rangle,
$
which is the coefficient of $\kschur_\lambda$ in $\Strong_{\mu/\nu}$.
\end{proof}

Consequently, we obtain an explicit combinatorial description of the skew
$k$-Schur function $\kschur_{\lambda/\mu}$ since the strong Schur function
$\Strong_{\mu/\nu}$ has an explicit combinatorial description in terms of
``strong tableaux'' (see \cite{LLMS10} for details).

\bibliographystyle{halpha}
\bibliography{references} 

\end{document}